\documentclass[12pt]{amsart}
\usepackage{amssymb}
\usepackage{hyperref}

\newtheorem{theorem}{Theorem}[section]
\newtheorem{definition}[theorem]{Definition}
\newtheorem{proposition}[theorem]{Proposition}
\newtheorem{lemma}[theorem]{Lemma}

\begin{document}

\title{The glow of Fourier matrices: universality and fluctuations}

\author{Teodor Banica}
\address{T.B.: Department of Mathematics, Cergy-Pontoise University, 95000 Cergy-Pontoise, France. {\tt teo.banica@gmail.com}}


\begin{abstract}
The glow of an Hadamard matrix $H\in M_N(\mathbb C)$ is the probability measure $\mu\in\mathcal P(\mathbb C)$ describing the distribution of $\varphi(a,b)=<a,Hb>$, where $a,b\in\mathbb T^N$ are random. We prove that $\varphi/N$ becomes complex Gaussian with $N\to\infty$, and that the universality holds as well at order 2. In the case of a Fourier matrix, $F_G\in M_N(\mathbb C)$ with $|G|=N$, the universality holds up to order 4, and the fluctuations are encoded by certain subtle integrals, which appear in connection with several Hadamard-related questions. In the Walsh matrix case, $G=\mathbb Z_2^n$, we conjecture that the glow is polynomial in $N=2^n$.
\end{abstract}

\maketitle

\tableofcontents

\section*{Introduction}

A complex Hadamard matrix is a matrix $H\in M_N(\mathbb T)$, whose rows are pairwise orthogonal. The basic example is the Fourier matrix, $F_N=(w^{ij})$ with $w=e^{2\pi i/N}$. More generally, associated to a finite abelian group $G=\mathbb Z_{N_1}\times\ldots\times\mathbb Z_{N_k}$ is its Fourier matrix $F_G=F_{N_1}\otimes\ldots\otimes F_{N_k}$, which is a complex Hadamard matrix. In general, the complex Hadamard matrices are known to appear in connection with a wide array of questions, mainly coming from operator algebras and quantum physics. See \cite{ben}, \cite{jsu}, \cite{pop}, \cite{tzy}.

Two such matrices $H,K\in M_N(\mathbb T)$ are called equivalent if one can pass from one to the other by permuting rows and columns, or by multiplying the rows and columns by numbers $z\in\mathbb T$. In the case where both Hadamard matrices are binary, $H,K\in M_N(\pm1)$, it is customary to use for the equivalence relation binary scalars only, $z\in\{\pm1\}$.

The glow is an analytic invariant introduced in \cite{ban}, inspired from the Gale-Berlekamp game \cite{fsl}, \cite{rvi}, and from the notion of numerical range \cite{cgl}, \cite{dgh}. The glow of $H\in M_N(\mathbb T)$ is by definition the complex probability measure $\mu\in\mathcal P(\mathbb C)$ describing the distribution of the total sum of the entries, $\Omega=\sum_{ij}H_{ij}$, over the equivalence class of $H$.

In order to understand this notion, let us first consider a binary matrix $H\in M_N(\pm1)$, and define $\mu$ as above, but with respect to the usual equivalence relation for binary matrices. It is useful to think of $\mu$ as being the ``glow'' of the matrix, in the following way. Assume that we have a square city, with $N$ horizontal streets and $N$ vertical streets, and with street lights at each crossroads. When evening comes the lights are switched on at the positions $(i,j)$ where $H_{ij}=1$, and then, all night long, they are randomly switched on and off, with the help of $2N$ master switches, one at the end of each street:
$$\begin{matrix}
\to&&\diamondsuit&\diamondsuit&\diamondsuit&\diamondsuit\\
\to&&\diamondsuit&\times&\diamondsuit&\times\\
\to&&\diamondsuit&\diamondsuit&\times&\times\\
\to&&\diamondsuit&\times&\times&\diamondsuit\\
\\
&&\uparrow&\uparrow&\uparrow&\uparrow
\end{matrix}$$ 

With this picture in mind, $\mu$ describes indeed the glow of the city.

Now back to the complex case, all the scalars will now belong to $\mathbb T$ instead of $\pm1$, but the above interpretation will somehow subsist, and this is why we call our invariant ``glow''. Observe that the glow is by construction rotationally invariant, so we are in fact basically interested in computing a probability measure supported by $\mathbb R_+$.

As already mentioned, there are some obvious connections with the Gale-Berlekamp game \cite{fsl}, \cite{rvi}, and with the notion of numerical range \cite{cgl}, \cite{dgh}. Yet another motivation comes from the operator algebra problematics in \cite{jsu}, \cite{pop}. Indeed, the spectral measure $\eta$ of the subfactor associated to $H\in M_N(\mathbb T)$ depends as well only on the equivalence class of $H$, and one may wonder whether is there is a deeper relation between $\mu,\eta$.

In order to further discuss the motivations, we will need the explicit formula of the moments of $\Omega=\sum_{ij}H_{ij}$. This formula, obtained by M\"obius inversion, is as follows:
$$\int_{\mathbb T^N\times\mathbb T^N}|\Omega|^{2p}=\sum_{\pi\in P(p)}K(\pi)N^{|\pi|}I(\pi)$$

Here $P(p)$ is the set of partitions of $\{1,\ldots,p\}$, and for $\pi\in P(p)$ we denote by $|\pi|$ the number of blocks, and we set $K(\pi)=\sum_{\sigma\in P(p)}\mu(\pi,\sigma)\binom{r}{\sigma}$, where $\mu$ is the M\"obius function. Regarding now $I(\pi)$, which is the key quantity in the above formula, this is: 
$$I(\pi)=\frac{1}{N^{|\pi|}}\sum_{[i]=[j]}<H_\pi(i),H_\pi(j)>$$

Here $[i]$ is the set with repetitions associated to a multi-index $i$, and we use the notation $H_\pi(i)=\bigotimes_{\beta\in\pi}\prod_{r\in\beta}H_{i_r}$, where $H_x\in\mathbb T^N$ are the rows of $H$. As a basic example, in the Fourier matrix case $H=F_G$, with $|G|=N$, these quantities are as follows:
$$I(\pi)=\int_{\mathbb T^N}\left(\prod_{\beta\in\pi}\sum_{\Sigma x_i=\Sigma y_i}\frac{a_{x_1}\ldots a_{x_{|\beta|}}}{a_{y_1}\ldots a_{y_{|\beta|}}}\right)da$$

With these formulae in hand, we can now explain our precise motivations for investigating the glow. We have in fact two main motivations, as follows:
\begin{enumerate}
\item Invariant theory. The moment formula shows that $\mu$ depends only on the quantities $<H_\pi(i),H_\pi(j)>$, and the same is known from \cite{bfs} to hold for the spectral measure $\eta$. Thus, we have evidence for a deeper relation between $\mu,\eta$.

\item Counting problems. In the Fourier matrix case, $H=F_G$ with $|G|=N$, the above quantities $I(\pi)$ are those coming from the work in \cite{bcs}, \cite{bns}, which control the number of $G$-patterned complex Hadamard matrices.
\end{enumerate}

Summarizing, we have reasons to believe that $\mu$ is an interesting invariant, and that its exact computation for the Fourier matrices $F_G$ is of particular interest. In this paper we will perform an asymptotic study of $\mu$, our conclusions being as follows:
\begin{enumerate}
\item $\Omega/N$ becomes complex Gaussian in the $N\to\infty$ limit.

\item A universality result holds as well at order $2$.

\item Within the class $\{F_G\}$, the universality holds up to order $4$.
\end{enumerate}

Perhaps the most surprising finding in this series is the last one. Here is the result, and we refer to the body of the paper for the precise statement:

\bigskip

\noindent {\bf Theorem.} {\em For a Fourier matrix $F_G$, with $|G|=N$, we have
$$\frac{1}{p!}\int_{\mathbb T^N\times\mathbb T^N}\left(\frac{|\Omega|}{N}\right)^{2p}=1+K_1N^{-1}+K_2N^{-2}+K_3N^{-3}+O(N^{-4})$$
with $K_1,K_2,K_3$ being certain polynomials in $p$, independent of $N$ and $G$.}

\bigskip

Regarding now the proof, this is based on the moment formula given above, and on a number of computations and estimates regarding the integrals $I(\pi)$. We believe that these computations and estimates can be of use in connection with the above-mentioned motivations, but for the moment we have no further results. Let us mention however that our computer simulations suggest to first look in detail at the Walsh matrix case.

The paper is organized as follows: 1 is a preliminary section, in 2-3 we state and prove our main results, and 4 contains a few concluding remarks.

\medskip

\noindent {\bf Acknowledgements.} I would like to thank Ion Nechita and Jean-Marc Schlenker for interesting discussions and recent joint work on related questions, and Beno\^it Collins for some help with a number of analytic issues.

\section{The binary glow}

An Hadamard matrix is a square matrix $H\in M_N(\pm 1)$, whose rows are pairwise orthogonal. The size of such a matrix must be $N=2$ or $N\in 4\mathbb N$. See \cite{sya}.

These matrices are usually taken under the following equivalence relation:

\begin{definition}
$H,K\in M_N(\pm 1)$ are called equivalent if one can pass from one to the other by permuting rows and columns, or switching signs on rows and columns.
\end{definition}

As explained in the introduction, we are interested in the total sum of the entries. Since this number is invariant under permutations of rows and columns, we can restrict attention to the matrices $\widetilde{H}\simeq H$ obtained by switching signs on rows and columns. More precisely, let $(a,b)\in\mathbb Z_2^N\times\mathbb Z_2^N$, and consider the following matrix:
$$\widetilde{H}_{ij}=a_ib_jH_{ij}$$

We will regard the sum of entries of $\widetilde{H}$ as a random variable, as follows:

\begin{definition}
Let $H\in M_N(\pm 1)$ be an Hadamard matrix.
\begin{enumerate}
\item We define $\varphi:\mathbb Z_2^N\times\mathbb Z_2^N\to\mathbb Z$ by $\varphi(a,b)=\sum_{ij}a_ib_jH_{ij}$.

\item We let $\mu$ be the probability measure on $\mathbb Z$ given by $\mu(\{k\})=P(\varphi=k)$.
\end{enumerate}
\end{definition}

In this definition $P$ denotes the probability with respect to the uniform measure on the group $\mathbb Z_2^N\times\mathbb Z_2^N$. In other words, we regard $\varphi$ as a random variable over this group, and we denote by $\mu$ the distribution of this random variable:
$$\mu(\{k\})=\frac{1}{4^N}\#\left\{(a,b)\in \mathbb Z_2^N\times\mathbb Z_2^N\Big|\varphi(a,b)=k\right\}$$

As explained in the introduction, $\mu$ can be thought of as being the ``glow'' of the matrix. In order to compute the glow, it is useful to have in mind the following picture:
$$\begin{matrix}
&&b_1&\ldots&b_N\\
&&\downarrow&&\downarrow\\
(a_1)&\to&H_{11}&\ldots&H_{1N}&\Rightarrow&S_1\\
\vdots&&\vdots&&\vdots&&\vdots\\
(a_N)&\to&H_{N1}&\ldots&H_{NN}&\Rightarrow&S_N
\end{matrix}$$

Here the columns of $H$ have been multiplied by the entries of the horizontal switching vector $b$, the resulting sums on rows are denoted $S_1,\ldots,S_N$, and the vertical switching vector $a$ still has to act on these sums, and produce the glow component at $b$.

\begin{proposition}
The glow of a matrix $H\in M_N(\pm 1)$ is given by
$$\mu=\frac{1}{2^N}\sum_{b\in\mathbb Z_2^N}\beta_1(c_1)*\ldots*\beta_N(c_N)$$
where $\beta_r(c)=\left(\frac{\delta_r+\delta_{-r}}{2}\right)^{*c}$, and $c_r=\#\{r\in|S_1|,\ldots,|S_N|\}$, with $S=Hb$.
\end{proposition}

\begin{proof}
We use the interpretation of the glow which was explained above. So, consider the decomposition of the glow over $b$ components:
$$\mu=\frac{1}{2^N}\sum_{b\in\mathbb Z_2^N}\mu_b$$

With the notation $S=Hb$, the numbers $S_1,\ldots,S_N$ are the sums on the rows of the matrix $\widetilde{H}_{ij}=H_{ij}a_ib_j$. Thus the glow components are given by:
$$\mu_b=law\left(\pm S_1\pm S_2\ldots\pm S_N\right)$$

By permuting now the sums on the right, we have the following formula:
$$\mu_b=law\big(\underbrace{\pm 0\ldots\pm 0}_{c_0}\underbrace{\pm 1\ldots\pm 1}_{c_1}\ldots\ldots\underbrace{\pm N\ldots\pm N}_{c_N}\big)$$

Now since the $\pm$ variables each follow a Bernoulli law, and these Bernoulli laws are independent, we obtain a convolution product as in the statement.
\end{proof}

We will need the following elementary lemma:

\begin{lemma}
Let $H\in M_N(\pm1)$ be an Hadamard matrix of order $N\geq 4$. 
\begin{enumerate}
\item The sums of entries on rows $S_1,\ldots,S_N$ are even, and equal modulo $4$.

\item If the sums on the rows $S_1,\ldots,S_N$ are all $0$ modulo $4$, then the number of rows whose sum is $4$ modulo $8$ is odd for $N=4(8)$, and even for $N=0(8)$.
\end{enumerate}
\end{lemma}

\begin{proof}
(1) Let us pick two rows of our matrix, and then permute the columns such that these two rows look as follows:
$$\begin{pmatrix}
+\ldots+&+\ldots+&-\ldots-&-\ldots-\\
\underbrace{+\ldots+}_a&\underbrace{-\ldots-}_b&\underbrace{+\ldots+}_c&\underbrace{-\ldots-}_d
\end{pmatrix}$$

We have $a+b+c+d=N$, and by orthogonality $a+d=b+c$, so $a+d=b+c=\frac{N}{2}$. Now since $N/2$ is even, we conclude that $b=c(2)$, and this gives the result.

(2) In the case where $H$ is ``row-dephased'', in the sense that its first row consists of $1$ entries only, the row sums are $N,0,\ldots,0$, and so the result holds. In general now, by permuting the columns we can assume that our matrix looks as follows:
$$H=\begin{pmatrix}+\ldots+&-\ldots-\\ \underbrace{\vdots}_x&\underbrace{\vdots}_y\end{pmatrix}$$

We have $x+y=N=0(4)$, and since the first row sum $S_1=x-y$ is by assumption 0 modulo 4, we conclude that $x,y$ are even. In particular, since $y$ is even, the passage from $H$ to its row-dephased version $\widetilde{H}$ can be done via $y/2$ double sign switches.

Now, in view of the above, it is enough to prove that the conclusion in the statement is stable under a double sign switch. So, let $H\in M_N(\pm1)$ be Hadamard, and let us perform to it a double sign switch, say on the first two columns. Depending on the values of the entries on these first two columns, the total sums on the rows change as follows:
\begin{eqnarray*}
\begin{pmatrix}+&+&\ldots&\ldots\end{pmatrix}&:&S\to S-4\\
\begin{pmatrix}+&-&\ldots&\ldots\end{pmatrix}&:&S\to S\\
\begin{pmatrix}-&+&\ldots&\ldots\end{pmatrix}&:&S\to S\\
\begin{pmatrix}-&-&\ldots&\ldots\end{pmatrix}&:&S\to S+4
\end{eqnarray*}

We can see that the changes modulo 8 of the row sum $S$ occur precisely in the first and in the fourth case. But, since the first two columns of our matrix $H\in M_N(\pm1)$ are orthogonal, the total number of these cases is even, and this finishes the proof.
\end{proof}

Observe that Proposition 1.3 and Lemma 1.4 (1) show that the glow of an Hadamard matrix of order $N\geq 4$ is supported by $4\mathbb Z$. With this remark in hand, we have:

\begin{proposition}
Let $H\in M_N(\pm1)$ be an Hadamard matrix of order $N\geq 4$, and denote by $\mu^{even},\mu^{odd}$ the mass one-rescaled restrictions of $\mu\in\mathcal P(4\mathbb Z)$ to $8\mathbb Z,8\mathbb Z+4$.
\begin{enumerate}
\item At $N=0(8)$ we have $\mu=\frac{3}{4}\mu^{even}+\frac{1}{4}\mu^{odd}$.

\item At $N=4(8)$ we have $\mu=\frac{1}{4}\mu^{even}+\frac{3}{4}\mu^{odd}$.
\end{enumerate}
\end{proposition}

\begin{proof}
We use the glow decomposition over $b$ components, from Proposition 1.3:
$$\mu=\frac{1}{2^N}\sum_{b\in\mathbb Z_2^N}\mu_b$$

The idea is that the decomposition formula in the statement will occur over averages of the following type, over truncated sign vectors $c\in\mathbb Z_2^{N-1}$:
$$\mu'_c=\frac{1}{2}(\mu_{+c}+\mu_{-c})$$

Indeed, we know from Lemma 1.4 (1) that modulo 4, the sums on rows are either $0,\ldots,0$ or $2,\ldots,2$. Now since these two cases are complementary when pairing switch vectors $(+c,-c)$, we can assume that we are in the case $0,\ldots,0$ modulo 4. 

Now by looking at this sequence modulo 8, and letting $x$ be the number of 4 components, so that the number of 0 components is $N-x$, we have:
$$\frac{1}{2}(\mu_{+c}+\mu_{-c})=\frac{1}{2}\left(law(\underbrace{\pm0\ldots\pm 0}_{N-x}\underbrace{\pm4\ldots\pm 4}_x)+law(\underbrace{\pm 2\ldots\pm 2}_N)\right)$$

Now by using Lemma 1.4 (2), the first summand splits $1-0$ or $0-1$ on $8\mathbb Z,8\mathbb Z+4$, depending on the class of $N$ modulo 8. As for the second summand, since $N$ is even this always splits $\frac{1}{2}-\frac{1}{2}$ on $8\mathbb Z,8\mathbb Z+4$. So, by making the average we obtain either a $\frac{3}{4}-\frac{1}{4}$ or a $\frac{1}{4}-\frac{3}{4}$ splitting on $8\mathbb Z,8\mathbb Z+4$, depending on the class of $N$ modulo 8, as claimed.
\end{proof}

Our various computer simulations suggest that the measures $\mu^{even},\mu^{odd}$ don't have further general algebraic properties. Analytically speaking now, we have:

\begin{theorem}
The binary glow moments of $H\in M_N(\pm1)$ are given by:
$$\int_{\mathbb Z_2^N\times\mathbb Z_2^N}\left(\frac{\Omega}{N}\right)^{2p}=(2p)!!+O(N^{-1})$$
In particular the variable $\Omega/N$ becomes Gaussian in the $N\to\infty$ limit.
\end{theorem}

\begin{proof}
Let $P_{even}(r)\subset P(r)$ be the set of partitions of $\{1,\ldots,r\}$ having all blocks of even size. The moments of the variable $\Omega=\sum_{ij}a_ib_jH_{ij}$ are then given by:
\begin{eqnarray*}
\int_{\mathbb Z_2^N\times\mathbb Z_2^N}\Omega^r
&=&\sum_{ix}H_{i_1x_1}\ldots H_{i_rx_r}\int_{\mathbb Z_2^N}a_{i_1}\ldots a_{i_r}\int_{\mathbb Z_2^N}b_{x_1}\ldots b_{x_r}\\
&=&\sum_{\pi,\sigma\in P_{even}(r)}\sum_{\ker i=\pi,\ker x=\sigma}H_{i_1x_1}\ldots H_{i_rx_r}
\end{eqnarray*}

Thus the moments decompose over partitions $\pi\in P_{even}(r)$, with the contributions being obtained by integrating the following quantities:
$$C(\sigma)=\sum_{\ker x=\sigma}\sum_iH_{i_1x_1}\ldots H_{i_rx_r}\cdot a_{i_1}\ldots a_{i_r}$$

Now by M\"obius inversion, we obtain a formula as follows:
$$\int_{\mathbb Z_2^N\times\mathbb Z_2^N}\Omega^r=\sum_{\pi\in P_{even}(r)}K(\pi)N^{|\pi|}I(\pi)$$

Here $K(\pi)=\sum_{\sigma\in P_{even}(r)}\mu(\pi,\sigma)$, where $\mu$ is the M\"obius function of $P_{even}(r)$, and $I(\pi)=\sum_i\prod_{b\in\pi}\frac{1}{N}\left\langle\prod_{r\in b}H_{i_r},1\right\rangle$, where $H_1,\ldots,H_N\in\mathbb Z_2^N$ are the rows of $H$.

With this formula in hand, the first assertion follows, because the biggest elements of the lattice $P_{even}(2p)$ are the $(2p)!!$ partitions consisting of $p$ copies of a $2$-block. 

As for the second assertion, this follows from the formula in the statement, and from the fact that the glow of $H\in M_N(\pm1)$ is real, and symmetric with respect to $0$.
\end{proof}

\section{Complex matrices}

In this section and in the next one we discuss the complex case, which is the one that we are truly interested in. We will use inspiration from section 1.

We recall that a complex Hadamard matrix is a matrix $H\in M_N(\mathbb T)$, where $\mathbb T$ is the unit circle in the complex plane, whose rows are pairwise orthogonal. Two such matrices $H,K$ are called equivalent if one can pass from one to the other by permuting the rows and columns, or by multiplying the rows and columns by numbers in $\mathbb T$. See \cite{tzy}.

As explained in the introduction, we are interested in the following invariant:

\begin{definition}
The glow of $H\in M_N(\mathbb T)$ is the probability measure $\mu\in\mathcal P(\mathbb C)$ given by:
$$\int_\mathbb C\varphi(x)d\mu(x)=\int_{\mathbb T^N\times\mathbb T^N}\varphi\left(\sum_{ij}a_ib_jH_{ij}\right)d(a,b)$$
That is, $\mu$ is the law of the variable $\Omega=\sum_{ij}H_{ij}$, over the equivalence class of $H$.
\end{definition}

As a first observation, since $\mu$ is invariant under rotations, we are in fact interested in computing a certain measure $\mu^+$ supported by $\mathbb R_+$. More precisely, if we denote by $\times$ the multiplicative convolution, and by $\varepsilon$ the uniform measure on $\mathbb T$, then the probability distributions $\mu,\mu^+$ of $\Omega,|\Omega|$ over $\mathbb T^N\times\mathbb T^N$ are related by the formula $\mu=\varepsilon\times\mu^+$.

We develop now some moment machinery. Let $P(p)$ be the set of partitions of $\{1,\ldots,p\}$, with its standard order relation $\leq$, which is such that $\sqcap\!\!\sqcap\ldots\leq\pi\leq|\ |\ldots|$, for any $\pi\in P(p)$. We denote by $\mu(\pi,\sigma)$ the associated M\"obius function, given by:
$$\mu(\pi,\sigma)=\begin{cases}
1&{\rm if}\ \pi=\sigma\\
-\sum_{\pi\leq\tau<\sigma}\mu(\pi,\tau)&{\rm if}\ \pi<\sigma\\
0&{\rm if}\ \pi\not\leq\sigma
\end{cases}$$

For $\pi\in P(p)$ we set $\binom{p}{\pi}=\binom{p}{b_1\ldots b_{|\pi|}}=\frac{p!}{b_1!\ldots b_{|\pi|}!}$, where $b_1,\ldots,b_{|\pi|}$ are the block lenghts. Finally, we use the following notation, where $H_1,\ldots,H_N\in\mathbb T^N$ are the rows of $H$:
$$H_\pi(i)=\bigotimes_{\beta\in\pi}\prod_{r\in\beta}H_{i_r}$$

With these notations, we have the following result:

\begin{proposition}
The glow moments of a matrix $H\in M_N(\mathbb T)$ are given by
$$\int_{\mathbb T^N\times\mathbb T^N}|\Omega|^{2p}=\sum_{\pi\in P(p)}K(\pi)N^{|\pi|}I(\pi)$$
where $K(\pi)=\sum_{\sigma\in P(p)}\mu(\pi,\sigma)\binom{p}{\sigma}$ and $I(\pi)=\frac{1}{N^{|\pi|}}\sum_{[i]=[j]}<H_\pi(i),H_\pi(j)>$.
\end{proposition}

\begin{proof}
The moments are given by the following formula:
\begin{eqnarray*}
\int_{\mathbb T^N\times\mathbb T^N}|\Omega|^{2p}
&=&\int_{\mathbb T^N\times\mathbb T^N}\Big|\sum_{ij}H_{ix}a_ib_x\Big|^{2p}
=\int_{\mathbb T^N\times\mathbb T^N}\left(\sum_{ijxy}\frac{H_{ix}}{H_{jy}}\cdot\frac{a_ib_x}{a_jb_y}\right)^p\\
&=&\sum_{ijxy}\frac{H_{i_1x_1}\ldots H_{i_px_p}}{H_{j_1y_1}\ldots H_{j_py_p}}\int_{\mathbb T^N}\frac{a_{i_1}\ldots a_{i_p}}{a_{j_1}\ldots a_{j_p}}\,da\int_{\mathbb T^N}\frac{b_{x_1}\ldots b_{x_p}}{b_{y_1}\ldots b_{y_p}}\,db\\
&=&\sum_{[i]=[j],[x]=[y]}\frac{H_{i_1x_1}\ldots H_{i_px_p}}{H_{j_1y_1}\ldots H_{j_py_p}}
\end{eqnarray*}

With $\sigma=\ker x,\rho=\ker y$, we deduce that the moments of $|\Omega|^2$ decompose over partitions, $\int_{\mathbb T^N\times\mathbb T^N}|\Omega|^{2p}=\int_{\mathbb T^N}\sum_{\sigma,\rho\in P(p)}C(\sigma,\rho)$, with the contributions being as follows:
$$C(\sigma,\rho)=\sum_{\ker x=\sigma,\ker y=\rho}\delta_{[x],[y]}\sum_{ij}
\frac{H_{i_1x_1}\ldots H_{i_px_p}}{H_{j_1y_1}\ldots H_{j_py_p}}
\cdot\frac{a_{i_1}\ldots a_{i_p}}{a_{j_1}\ldots a_{j_p}}$$

We have $C(\sigma,\rho)=0$ unless $\sigma\sim\rho$, in the sense that $\sigma,\rho$ must have the same block structure. The point now is that the sums of type $\sum_{\ker x=\sigma}$ can be computed by using the M\"obius inversion formula. We obtain a formula as follows:
$$C(\sigma,\rho)=\delta_{\sigma\sim\rho}\sum_{\pi\leq\sigma}\mu(\pi,\sigma)\prod_{\beta\in\pi}C_{|\beta|}(a)$$

Here the functions on the right are by definition given by:
\begin{eqnarray*}
C_r(a)
&=&\sum_x\sum_{ij}\frac{H_{i_1x}\ldots H_{i_rx}}{H_{j_1x}\ldots H_{j_rx}}\cdot\frac{a_{i_1}\ldots a_{i_r}}{a_{j_1}\ldots a_{j_r}}\\
&=&\sum_{ij}<H_{i_1}\ldots H_{i_r},H_{j_1}\ldots H_{j_r}>\cdot\frac{a_{i_1}\ldots a_{i_r}}{a_{j_1}\ldots a_{j_r}}
\end{eqnarray*}

Now since there are $\binom{p}{\sigma}$ partitions having the same block structure as $\sigma$, we obtain:
\begin{eqnarray*}
\int_{\mathbb T^N\times\mathbb T^N}|\Omega|^{2p}
&=&\int_{\mathbb T^N}\sum_{\pi\in P(p)}\left(\sum_{\sigma\sim\rho}\sum_{\mu\leq\sigma}\mu(\pi,\sigma)\right)\prod_{\beta\in\pi}C_{|\beta|}(a)\\
&=&\sum_{\pi\in P(p)}\left(\sum_{\sigma\in P(p)}\mu(\pi,\sigma)\binom{p}{\sigma}\right)\int_{\mathbb T^N}\prod_{\beta\in\pi}C_{|\beta|}(a)
\end{eqnarray*}

But this gives the formula in the statement, and we are done.
\end{proof}

Let us discuss now the asymptotic behavior of the glow. For this purpose, we first study the coefficients $K(\pi)$ in Proposition 2.2. We have here:

\begin{lemma}
$K(\pi)=\sum_{\pi\leq\sigma}\mu(\pi,\sigma)\binom{p}{\sigma}$ has the following properties:
\begin{enumerate}
\item $\widetilde{K}(\pi)=\frac{K(\pi)}{p!}$ is multiplicative: $\widetilde{K}(\pi\pi')=\widetilde{K}(\pi)\widetilde{K}(\pi')$.

\item $K(\sqcap\!\!\sqcap\ldots\sqcap)=\sum_{\sigma\in P(p)}(-1)^{|\sigma|-1}(|\sigma|-1)!\binom{p}{\sigma}$.

\item $K(\sqcap\!\!\sqcap\ldots\sqcap)=\sum_{r=1}^p(-1)^{r-1}(r-1)!C_{pr}$, where $C_{pr}=\sum_{p=a_1+\ldots+a_r}\binom{p}{a_1,\ldots,a_r}^2$.
\end{enumerate}
\end{lemma}

\begin{proof}
(1) We use the fact that $\mu(\pi\pi',\sigma\sigma')=\mu(\pi,\sigma)\mu(\pi',\sigma')$, which is a well-known property of the M\"obius function, which can be proved by recurrence. Now if $b_1,\ldots,b_s$ and $c_1,\ldots,c_t$ are the block lengths of $\sigma,\sigma'$, we obtain, as claimed:
\begin{eqnarray*}
\widetilde{K}(\pi\pi')
&=&\sum_{\pi\pi'\leq\sigma\sigma'}\mu(\pi\pi',\sigma\sigma')\cdot\frac{1}{b_1!\ldots b_s!}\cdot\frac{1}{c_1!\ldots c_t!}\\
&=&\sum_{\pi\leq\sigma,\pi'\leq\sigma'}\mu(\pi,\sigma)\mu(\pi',\sigma')\cdot\frac{1}{b_1!\ldots b_s!}\cdot\frac{1}{c_1!\ldots c_t!}\\
&=&\widetilde{K}(\pi)\widetilde{K}(\pi')
\end{eqnarray*}

(2) We use here the formula $\mu(\sqcap\!\!\sqcap\ldots\sqcap,\sigma)=(-1)^{|\sigma|-1}(|\sigma|-1)!$, which once again is well-known, and can be proved by recurrence on $|\sigma|$. We obtain, as claimed:
$$K(\sqcap\!\!\sqcap\ldots\sqcap)=\sum_{\sigma\in P(p)}\mu(\sqcap\!\!\sqcap\ldots\sqcap,\sigma)\binom{p}{\sigma}=\sum_{\sigma\in P(p)}(-1)^{|\sigma|-1}(|\sigma|-1)!\binom{p}{\sigma}$$

(3) By using the formula in (2), and summing over $r=|\sigma|$, we obtain:
$$K(\sqcap\!\!\sqcap\ldots\sqcap)
=\sum_{r=1}^p(-1)^{r-1}(r-1)!\sum_{|\sigma|=r}\binom{p}{\sigma}$$

Now if we denote by $a_1,\ldots,a_r$ with $a_i\geq1$ the block lengths of $\sigma$, then $\binom{p}{\sigma}=\binom{p}{a_1,\ldots,a_r}$. On the other hand, given $a_1,\ldots,a_r\geq1$ with $a_1+\ldots+a_r=p$, there are exactly $\binom{p}{a_1,\ldots,a_r}$ partitions $\sigma$ having these numbers as block lengths, and this gives the result.
\end{proof}

Now let us take a closer look at the integrals $I(\pi)$. We have here:

\begin{lemma}
Consider the one-block partition $\sqcap\!\!\sqcap\ldots\sqcap\in P(p)$.
\begin{enumerate}
\item $I(\sqcap\!\!\sqcap\ldots\sqcap)=\#\{i,j\in\{1,\ldots,N\}^p|[i]=[j]\}$.

\item $I(\sqcap\!\!\sqcap\ldots\sqcap)=\int_{\mathbb T^N}|\sum_ia_i|^{2p}da$.

\item $I(\sqcap\!\!\sqcap\ldots\sqcap)=\sum_{\sigma\in P(p)}\binom{p}{\sigma}\frac{N!}{(N-|\sigma|)!}$.

\item $I(\sqcap\!\!\sqcap\ldots\sqcap)=\sum_{r=1}^{p-1}C_{pr}\frac{N!}{(N-r)!}$, where $C_{pr}=\sum_{p=b_1+\ldots+b_r}\binom{p}{b_1,\ldots,b_r}^2$.
\end{enumerate}
\end{lemma}

\begin{proof}
(1)  This follows indeed from the following computation:
$$I(\sqcap\!\!\sqcap\ldots\sqcap)=\sum_{[i]=[j]}\frac{1}{N}<H_{i_1}\ldots H_{i_r},H_{j_1}\ldots H_{j_r}>=\sum_{[i]=[j]}1$$

(2) This follows from the following computation:
$$\int_{\mathbb T^N}\left|\sum_ia_i\right|^{2p}=\int_{\mathbb T^N}\sum_{ij}\frac{a_{i_1}\ldots a_{i_p}}{a_{j_1}\ldots a_{j_p}}da=\#\left\{i,j\Big|[i]=[j]\right\}$$

(3) If we let $\sigma=\ker i$ in the above formula of $I(\sqcap\!\!\sqcap\ldots\sqcap)$, we obtain:
$$I(\sqcap\!\!\sqcap\ldots\sqcap)=\sum_{\sigma\in P(p)}\#\left\{i,j\Big|\ker i=\sigma,[i]=[j]\right\}$$

Now since there are $\frac{N!}{(N-|\sigma|)!}$ choices for $i$, and then $\binom{p}{\sigma}$ for $j$, this gives the result.

(4) If we set $r=|\sigma|$, the formula in (3) becomes:
$$I(\sqcap\!\!\sqcap\ldots\sqcap)=\sum_{r=1}^{p-1}\frac{N!}{(N-r)!}\sum_{\sigma\in P(p),|\sigma|=r}\binom{p}{\sigma}$$

Now since there are exactly $\binom{p}{b_1,\ldots,b_r}$ permutations $\sigma\in P(p)$ having $b_1,\ldots,b_r$ as block lengths, the sum on the right equals $\sum_{p=b_1+\ldots+b_r}\binom{p}{b_1,\ldots,b_r}^2$, as claimed.
\end{proof}

In general, the integrals $I(\pi)$ can be estimated as follows:

\begin{lemma}
Let $H\in M_N(\mathbb T)$, having its rows pairwise orthogonal.
\begin{enumerate}
\item $I(|\,|\,\ldots|)=N^p$.

\item $I(|\,|\,\ldots|\ \pi)=N^aI(\pi)$, for any $\pi\in P(p-a)$.

\item $|I(\pi)|\lesssim p!N^p$, for any $\pi\in P(p)$. 
\end{enumerate}
\end{lemma}

\begin{proof}
(1) Since the rows of $H$ are pairwise orthogonal, we have:
$$I(|\,|\ldots|)=\sum_{[i]=[j]}\prod_{r=1}^p\delta_{i_r,j_r}=\sum_{[i]=[j]}\delta_{ij}=\sum_i1=N^p$$

(2) This follows by the same computation as the above one for (1).

(3) We have indeed the following estimate:
$$|I(\pi)|
\leq\sum_{[i]=[j]}\prod_{\beta\in\pi}1=\sum_{[i]=[j]}1=\#\left\{i,j\in\{1,\ldots,N\}\Big|[i]=[j]\right\}\simeq p!N^p$$

Thus we have obtained the formula in the statement, and we are done.
\end{proof}

We have now all needed ingredients for a universality result:

\begin{theorem}
The glow of a complex Hadamard matrix $H\in M_N(\mathbb T)$ is given by:
$$\frac{1}{p!}\int_{\mathbb T^N\times\mathbb T^N}\left(\frac{|\Omega|}{N}\right)^{2p}=1-\binom{p}{2}N^{-1}+O(N^{-2})$$
In particular, $\Omega/N$ becomes complex Gaussian in the $N\to\infty$ limit.
\end{theorem}

\begin{proof}
We use the moment formula in Proposition 2.2. By using Lemma 2.5 (3), we conclude that only the $p$-block and $(p-1)$-block partitions contribute at order 2, so:
$$\int_{\mathbb T^N\times\mathbb T^N}|\Omega|^{2p}=K(|\,|\ldots|)N^pI(|\,|\ldots|)+\binom{p}{2}K(\sqcap|\ldots|)N^{p-1}I(\sqcap|\ldots|)+O(N^{2p-2})$$

Now by dividing by $N^{2p}$ and then by using the various formulae in Lemma 2.3, Lemma 2.4 and Lemma 2.5 above, we obtain, as claimed:
$$\int_{\mathbb T^N\times\mathbb T^N}\left(\frac{|\Omega|}{N}\right)^{2p}
=p!-\binom{p}{2}\frac{p!}{2}\cdot\frac{2N-1}{N^2}+O(N^{-2})$$

Finally, since the law of $\Omega$ is invariant under centered rotations in the complex plane, this moment formula gives as well the last assertion.
\end{proof}

\section{Fourier matrices}

In this section we study the glow of an arbitrary Fourier matrix, $F=F_G$. We use the standard formulae $F_{ix}F_{iy}=F_{i,x+y}$, $\overline{F}_{ix}=F_{i,-x}$ and $\sum_xF_{ix}=N\delta_{i0}$. We first have:

\begin{proposition}
For a Fourier matrix $F_G$ we have
$$I(\pi)=\#\left\{i,j\Big|[i]=[j],\sum_{r\in\beta}i_r=\sum_{r\in\beta}j_r,\forall\beta\in\pi\right\}$$
with all the indices, and with the sums at right, taken inside $G$.
\end{proposition}

\begin{proof}
The basic components of the integrals $I(\pi)$ are given by:
$$\frac{1}{N}\left\langle\prod_{r\in\beta}F_{i_r},\prod_{r\in\beta}F_{j_r}\right\rangle=\frac{1}{N}\left\langle F_{\sum_{r\in\beta}i_r},F_{\sum_{r\in\beta}i_r}\right\rangle=\delta_{\sum_{r\in\beta}i_r,\sum_{r\in\beta}j_r}$$

But this gives the formula in the statement, and we are done.
\end{proof}

We have the following interpretation of the above integrals:

\begin{proposition}
For any partition $\pi$ we have the formula
$$I(\pi)=\int_{\mathbb T^N}\prod_{b\in\pi}\left(\frac{1}{N^2}\sum_{ij}|H_{ij}|^{2|\beta|}\right)da$$
where $H=FAF^*$, with $F=F_G$ and $A=diag(a_0,\ldots,a_{N-1})$.
\end{proposition}

\begin{proof}
We have the following computation:
\begin{eqnarray*}
H=F^*AF
&\implies&|H_{xy}|^2=\sum_{ij}\frac{F_{iy}F_{jx}}{F_{ix}F_{jy}}\cdot\frac{a_i}{a_j}\\
&\implies&|H_{xy}|^{2p}=\sum_{ij}\frac{F_{j_1x}\ldots F_{j_px}}{F_{i_1x}\ldots F_{i_px}}\cdot\frac{F_{i_1y}\ldots F_{i_py}}{F_{j_1y}\ldots F_{j_py}}\cdot\frac{a_{i_1}\ldots a_{i_p}}{a_{j_1}\ldots a_{j_p}}\\
&\implies&\sum_{xy}|H_{xy}|^{2p}=\sum_{ij}\left|<H_{i_1}\ldots H_{i_p},H_{j_1}\ldots H_{j_p}>\right|^2\cdot\frac{a_{i_1}\ldots a_{i_p}}{a_{j_1}\ldots a_{j_p}}
\end{eqnarray*}

But this gives the formula in the statement, and we are done.
\end{proof}

The above formula is interesting in connection with the considerations in \cite{bcs}, \cite{bns}, and with the general counting problematics for circulant Hadamard matrices \cite{haa}. See \cite{bns}.

Regarding now the glow estimates, we first have the following result:

\begin{lemma}
For $F_G$ we have the estimate
$$I(\pi)=b_1!\ldots b_{|\pi|}!N^p+O(N^{p-1})$$
where $b_1,\ldots,b_{|\pi|}$ with $b_1+\ldots+b_{|\pi|}=p$ are the block lengths of $\pi$.
\end{lemma}

\begin{proof}
With $\sigma=\ker i$ we obtain:
$$I(\pi)=\sum_{\sigma\in P(p)}\#\left\{i,j\Big|\ker i=\sigma,[i]=[j],\sum_{r\in\beta}i_r=\sum_{r\in\beta}j_r,\forall\beta\in\pi\right\}$$

Since there are $\frac{N!}{(N-|\sigma|)!}\simeq N^{|\sigma|}$ choices for $i$ satisfying $\ker i=\sigma$, and then there are $\binom{p}{\sigma}=O(1)$ choices for $j$ satisfying $[i]=[j]$, we conclude that the main contribution comes from $\sigma=|\,|\ldots|$, and so we have:
$$I(\pi)=\#\left\{i,j\Big|\ker i=|\,|\ldots|,[i]=[j],\sum_{r\in\beta}i_r=\sum_{r\in\beta}j_r,\forall\beta\in\pi\right\}+O(N^{p-1})$$

Now the condition $\ker i=|\,|\ldots|$ tells us that $i$ must have distinct entries, and there are $\frac{N!}{(N-p)!}\simeq N^p$ choices for such multi-indices $i$. Regarding now the indices $j$, the main contribution comes from those obtained from $i$ by permuting the entries over the blocks of $\pi$, and since there are $b_1!\ldots b_{|\pi|}!$ choices here, this gives the result.
\end{proof}

At the second order now, the estimate is as follows:

\begin{lemma}
For $F_G$ we have the formula
$$\frac{I(\pi)}{b_1!\ldots b_s!N^p}=1+\left(\sum_{i<j}\sum_{c\geq2}\binom{b_i}{c}\binom{b_j}{c}-\frac{1}{2}\sum_i\binom{b_i}{2}\right)N^{-1}+O(N^{-2})$$
where $b_1,\ldots,b_s$ being the block lengths of $\pi\in P(p)$.
\end{lemma}

\begin{proof}
Let us define the ``non-arithmetic'' part of $I(\pi)$ as follows: 
$$I^\circ(\pi)=\#\left\{i,j\Big|[i_r|r\in\beta]=[j_r|r\in\beta],\forall\beta\in\pi\right\}$$

We then have the following formula:
$$I^\circ(\pi)=\prod_{\beta\in\pi}\left\{i,j\in I^{|\beta|}\Big|[i]=[j]\right\}=\prod_{\beta\in\pi}I(\beta)$$

Also, Lemma 3.3 shows that we have the following estimate:
$$I(\pi)=I^\circ(\pi)+O(N^{p-1})$$

Our claim now is that we have the folowing formula:
$$\frac{I(\pi)-I^\circ(\pi)}{b_1!\ldots b_s!N^p}=\sum_{i<j}\sum_{c\geq2}\binom{b_i}{c}\binom{b_j}{c}N^{-1}+O(N^{-2})$$

Indeed, according to Lemma 3.3, we have a formula of the following type:
$$I(\pi)=I^\circ(\pi)+I^1(\pi)+O(N^{p-2})$$

More precisely, this formula holds indeed, with $I^1(\pi)$ coming from $i_1,\ldots,i_p$ distinct, $[i]=[j]$, and with one constraint of type $\sum_{r\in\beta}i_r=\sum_{j\in\beta}j_r$, with $[i_r|r\in\beta]\neq[j_r|r\in\beta]$. Now observe that for a two-block partition $\pi=(a,b)$ this constraint is implemented, up to permutations which leave invariant the blocks of $\pi$, as follows:
$$\begin{matrix}
i_1\ldots i_c&k_1\ldots k_{a-c}&&j_1\ldots j_c&l_1\ldots l_{a-c}\\
\underbrace{j_1\ldots j_c}_c&\underbrace{k_1\ldots k_{a-c}}_{a-c}&&\underbrace{i_1\ldots i_c}_c&\underbrace{l_1\ldots l_{a-c}}_{b-c}
\end{matrix}$$

Let us compute now $I^1(a,b)$. We cannot have $c=0,1$, and once $c\geq2$ is given, we have $\binom{a}{c},\binom{b}{c}$ choices for the positions of the $i,j$ variables in the upper row, then $N^{p-1}+O(N^{p-2})$ choices for the variables in the upper row, and then finally we have $a!b!$ permutations which can produce the lower row. We therefore obtain:
$$I^1(a,b)=a!b!\sum_{c\geq2}\binom{a}{c}\binom{b}{c}N^{p-1}+O(N^{p-2})$$

In the general case now, a similar discussion applies. Indeed, the constraint of type $\sum_{r\in\beta}i_r=\sum_{r\in\beta}j_r$ with $[i_r|r\in\beta]\neq[j_r|r\in\beta]$ cannot affect $\leq1$ blocks, because we are not in the non-arithmetic case, and cannot affect either $\geq3$ blocks, because affecting $\geq3$ blocks would require $\geq2$ constraints. Thus this condition affects exactly $2$ blocks, and if we let $i<j$ be the indices in $\{1,\ldots,s\}$ corresponding to these 2 blocks, we obtain:
$$I^1(\pi)=b_1!\ldots b_s!\sum_{i<j}\sum_{c\geq2}\binom{b_i}{c}\binom{b_j}{c}N^{p-1}+O(N^{p-2})$$

But this proves the above claim. Let us estimate now $I(\sqcap\!\!\sqcap\ldots\sqcap)$. We have:
\begin{eqnarray*}
I(\sqcap\!\!\sqcap\ldots\sqcap)
&=&p!\frac{N!}{(N-p)!}+\binom{p}{2}\frac{p!}{2}\cdot\frac{N!}{(N-p+1)!}+O(N^{p-2})\\
&=&p!N^r\left(1-\binom{p}{2}N^{-1}+O(N^{-2})\right)+\binom{p}{2}\frac{p!}{2}N^{p-1}+O(N^{p-2})\\
&=&p!N^p\left(1-\frac{1}{2}\binom{p}{2}N^{-1}+O(N^{-2})\right)
\end{eqnarray*}

Now by using the formula $I^\circ(\pi)=\prod_{\beta\in\pi}I(\beta)$, we obtain:
$$I^\circ(\pi)=b_1!\ldots b_s!N^p\left(1-\frac{1}{2}\sum_i\binom{b_i}{2}N^{-1}+O(N^{-2})\right)$$

By plugging this quantity into the above estimate, we obtain the result.
\end{proof}

In order to estimate glow, we will need the explicit formula of $I(\sqcap\sqcap)$:

\begin{lemma}
For $F_G$ with $G=\mathbb Z_{N_1}\times\ldots\times\mathbb Z_{N_k}$ we have the formula
$$I(\sqcap\sqcap)=N(4N^3-11N+2^e+7)$$
where $e\in\{0,1,\ldots,k\}$ is the number of even numbers among $N_1,\ldots,N_k$.
\end{lemma}

\begin{proof}
We use the fact that, when dealing with the conditions $\sum_{r\in\beta}i_r=\sum_{r\in\beta}j_r$ defining  the quantities $I(\pi)$, one can always erase some of the variables $i_r,j_r$, as to reduce to the ``purely arithmetic'' case, $\{i_r|r\in\beta\}\cap\{j_r|r\in\beta\}=\emptyset$. We have:
$$I(\sqcap\sqcap)=I^\circ(\sqcap\sqcap)+I^{ari}(\sqcap\sqcap)$$

Let us compute now $I^{ari}(\sqcap\sqcap)$. There are 3 contributions to this quantity, namely:

(1) \underline{Case $(^{iijj}_{jjii})$}, with $i\neq j$, $2i=2j$. Since $2(i_1,\ldots,i_k)=2(j_1,\ldots,j_k)$ corresponds to the collection of conditions $2i_r=2j_r$, inside $\mathbb Z_{N_r}$, which each have 1 or 2 solutions, depending on whether $N_r$ is odd or even, the contribution here is:
\begin{eqnarray*}
I^{ari}_1(\sqcap\sqcap)
&=&\#\{i\neq j|2i=2j\}\\
&=&\#\{i,j|2i=2j\}-\#\{i,j|i=j\}\\
&=&2^eN-N\\
&=&(2^e-1)N
\end{eqnarray*}

(2) \underline{Case $(^{iijk}_{jkii})$}, with $i,j,k$ distinct, $2i=j+k$. The contribution here is:
\begin{eqnarray*}
I^{ari}_2(\sqcap\sqcap)
&=&4\#\{i,j,k\ {\rm distinct}|2i=j+k\}\\
&=&4\#\{i\neq j|2i-j\neq i,j\}\\
&=&4\#\{i\neq j|2i\neq 2j\}\\
&=&4(\#\{i,j|i\neq j\}-\#\{i\neq j|2i=2j\})\\
&=&4(N(N-1)-(2^e-1)N)\\
&=&4N(N-2^e)
\end{eqnarray*}

(3) \underline{Case $(^{ijkl}_{klij})$}, with $i,j,k,l$ distinct, $i+j=k+l$. The contribution here is:
\begin{eqnarray*}
I^{ari}_3(\sqcap\sqcap)
&=&4\#\{i,j,k,l\ {\rm distinct}|i+j=k+l\}\\
&=&4\#\{i,j,k\ {\rm distinct}|i+j-k\neq i,j,k\}\\
&=&4\#\{i,j,k\ {\rm distinct}|i+j-k\neq k\}\\
&=&4\#\{i,j,k\ {\rm distinct}|i\neq 2k-j\}
\end{eqnarray*}

We can split this quantity over two cases, $2j\neq 2k$ and $2j=2k$, and we obtain:
\begin{eqnarray*}
I^{ari}_3(\sqcap\sqcap)
&=&4(\#\{i,j,k\ {\rm distinct}|2j\neq 2k,i\neq 2k-j\}\\
&&+\#\{i,j,k\ {\rm distinct}|2j=2k,i\neq 2k-j\})
\end{eqnarray*}

The point now is that in the first case, $2j\neq 2k$, the numbers $j,k,2k-j$ are distinct, while in the second case, $2j=2k$, we simply have $2k-j=j$. Thus, we obtain:
\begin{eqnarray*}
I^{ari}_3(\sqcap\sqcap)
&=&4\left(\sum_{j\neq k,2j\neq 2k}\#\{i|i\neq j,k,2k-j\}+\sum_{j\neq k,2j=2k}\#\{i|i\neq j,k\}\right)\\
&=&4(N(N-2^e)(N-3)+N(2^e-1)(N-2))\\
&=&4N(N(N-3)-2^e(N-3)+2^e(N-2)-(N-2))\\
&=&4N(N^2-4N+2^e+2)
\end{eqnarray*}

We can now compute the arithmetic part. This is given by:
\begin{eqnarray*}
I^{ari}(\sqcap\sqcap)
&=&(2^e-1)N+4N(N-2^e)+4N(N^2-4N+2^e+2)\\
&=&N(2^e-1+4(N-2^e)+4(N^2-4N+2^e+2))\\
&=&N(4N^2-12N+2^e+7)
\end{eqnarray*}

Thus the integral to be computed is given by:
\begin{eqnarray*}
I(\sqcap\sqcap)
&=&N^2(2N-1)^2+N(4N^2-12N+2^e+7)\\
&=&N(4N^3-4N^2+N+4N^2-12N+2^e+7)\\
&=&N(4N^3-11N+2^e+7)
\end{eqnarray*}

Thus we have reached to the formula in the statement, and we are done.
\end{proof}

We have the following asymptotic result:

\begin{theorem}
The glow of $F_G$, with $|G|=N$, is given by 
$$\frac{1}{p!}\int_{\mathbb T^N\times\mathbb T^N}\left(\frac{|\Omega|}{N}\right)^{2p}=1-K_1N^{-1}+K_2N^{-2}-K_3N^{-3}+O(N^{-4})$$ 
with $K_1=\binom{p}{2}$, $K_2=\binom{p}{2}\frac{3p^2+p-8}{12}$, $K_3=\binom{p}{3}\frac{p^3+4p^2+p-18}{8}$.
\end{theorem}

\begin{proof}
We use the quantities $\widetilde{K}(\pi)=\frac{K(\pi)}{p!},\widetilde{I}(\pi)=\frac{I(\pi)}{N^p}$, which are such that $\widetilde{K}(\pi|\ldots|)=\widetilde{K}(\pi),\widetilde{I}(\pi|\ldots|)=\widetilde{I}(\pi)$. In terms of $J(\sigma)=\binom{p}{\sigma}\widetilde{K}(\sigma)\widetilde{I}(\sigma)$, we have:
\begin{eqnarray*}
\frac{1}{p!}\int_{\mathbb T^N\times\mathbb T^N}|\Omega|^{2p}
&=&J(\emptyset)+N^{-1}J(\sqcap)+N^{-2}\left(J(\sqcap\!\sqcap)+J(\sqcap\sqcap)\right)\\
&&+N^{-3}\left(J(\sqcap\!\!\sqcap\!\!\sqcap)+J(\sqcap\!\!\sqcap\sqcap)+J(\sqcap\sqcap\sqcap)\right)+O(N^{-4})
\end{eqnarray*}

We have $\widetilde{K}_0=\widetilde{K}_1=1$, $\widetilde{K}_2=\frac{1}{2}-1=-\frac{1}{2}$, $\widetilde{K}_3=\frac{1}{6}-\frac{3}{2}+2=\frac{2}{3}$ and:
$$\widetilde{K}_4=\frac{1}{24}-\frac{4}{6}-\frac{3}{4}+\frac{12}{2}-6=-\frac{11}{8}$$

Regarding now the numbers $C_{pr}$ in Lemma 2.3, these are given by:
$$C_{p1}=1,C_{p2}=\frac{1}{2}\binom{2p}{p}-1,\ldots,C_{p,p-1}=\frac{p!}{2}\binom{p}{2},C_{pp}=p!$$

We deduce that $I(|)=N$, $I(\sqcap)=N(2N-1)$, $I(\sqcap\!\sqcap)=N(6N^2-9N+4)$ and:
$$I(\sqcap\!\!\sqcap\!\!\sqcap)=N(24N^3-72N^2+82N-33)$$

By using as well Lemma 3.4 and Lemma 3.5, we obtain the following formula:
\begin{eqnarray*}
\frac{1}{p!}\int_{\mathbb T^N\times\mathbb T^N}|\Omega|^{2p}
&=&1-\frac{1}{2}\binom{p}{2}(2N^{-1}-N^{-2})+\frac{2}{3}\binom{p}{3}(6N^{-2}-9N^{-3})+3\binom{p}{4}N^{-2}\\
&&-33\binom{p}{4}N^{-3}-40\binom{p}{5}N^{-3}-15\binom{p}{6}N^{-3}+O(N^{-4})
\end{eqnarray*}

But this gives the formulae of $K_1,K_2,K_3$ in the statement, and we are done.
\end{proof}

\section{Concluding remarks}

We have seen in this paper that the glow of the Fourier matrices, an invariant which is related to a wide array of Hadamard matrix questions, is complex Gaussian in the $N\to\infty$ limit, and that the universality holds in fact up to order 4. Any potential application, however, would require a much deeper understanding of the glow.

We believe that the formula in Theorem 3.6 should appear as truncation of the $N^{-1}$ expansion of the moment formula for some ``universal'' measure $\mu_N$. However, we do not know for the moment on how to approach $\mu_N$. One problem here is that we have two natural candidates for this measure, one coming from the Fourier matrix $F_N$, and the other one coming from the Walsh matrix $W_N$ (defined only for values $N=2^n$):
\begin{enumerate}
\item Regarding $F_N$, the formula in Lemma 3.5 shows that the next term $K_4$ depends on the parity of $N$. One could conjecture then that $K_4$ might be polynomial both in $N$ odd, and in $N$ even, but we have computer-assisted results showing that it is not so. It is not clear on how to advance in this direction.

\item Regarding $W_N$, where $N=2^n$ and the underlying group is $G=\mathbb Z_2^n$, here the numbers $C_I(J_1,\ldots,J_r)=\#\left\{(a_i)_{i\in I}\in G\ {\rm distinct}\Big|\sum_{j\in J_s}a_j=0,\forall s\right\}$ are polynomial in $N=2^n$, and this suggests that the integrals $I(\pi)$, and hence the glow, should be polynomial in $N$. However, we don't have a full proof of this fact.
\end{enumerate}

As a conclusion, the results in the present paper, along with the data coming from some extra computer simulations and computations, suggest the following key problem, that we would like to raise here: what is the glow of the Walsh matrices?

\end{document}